%% file: mcqmc-GKS-20171019.tex
\newif\ifspringer
\springerfalse
\ifspringer

    \documentclass{svmult}

    \smartqed
    \usepackage{mathptmx}       
    \usepackage{helvet}         
    \usepackage{courier}        
    \usepackage{type1cm}        

    \renewcommand{\email}[1]{\emailname: #1} 
    \renewenvironment{proof}{\noindent{\itshape Proof.}}{\smartqed\qed}
    \spdefaulttheorem{assumption}{Assumption}{\upshape \bfseries}{\itshape}
    \spdefaulttheorem{algo}{Algorithm}{\upshape \bfseries}{\itshape}

\else

    \documentclass{article}

    \usepackage[a4paper]{geometry}

    \newenvironment{acknowledgement}{\paragraph{Acknowledgement.}}{\par}
    \newcommand{\email}[1]{\texttt{#1}}
 
    \usepackage{amsthm}
    \theoremstyle{plain}
        \newtheorem{theorem}{Theorem}
        
        \newtheorem{lemma}{Lemma}

    \theoremstyle{definition}

    \theoremstyle{remark}

\fi

\usepackage{array,colortbl}
\usepackage{amsmath,amsfonts,amssymb,bm} 
\DeclareFontFamily{U}{mathx}{\hyphenchar\font45}
\DeclareFontShape{U}{mathx}{m}{n}{
      <5> <6> <7> <8> <9> <10>
      <10.95> <12> <14.4> <17.28> <20.74> <24.88>
      mathx10
      }{}
\DeclareSymbolFont{mathx}{U}{mathx}{m}{n}
\DeclareFontSubstitution{U}{mathx}{m}{n}
\DeclareMathAccent{\widecheck}      {0}{mathx}{"71}

\input{macros}

\DeclareSymbolFont{bbold}{U}{bbold}{m}{n}
\DeclareSymbolFontAlphabet{\mathbbold}{bbold}

\usepackage{microtype} 

\usepackage[colorlinks=true,linkcolor=black,citecolor=black,urlcolor=black]{hyperref}
\usepackage{bookmark}
\makeatletter 
  \providecommand*{\toclevel@author}{999}
  \providecommand*{\toclevel@title}{0}
\makeatother

\usepackage{accents} 
\newcommand{\GKSbell}{{\underaccent{\tilde}{\boldsymbol \ell}}} 
\newcommand{\GKSbe}{{\underaccent{\tilde}{\boldsymbol e}}}      
\begin{document}

\ifspringer

    \title*{Combining Sparse Grids, Multilevel MC and QMC
           for Elliptic PDEs with Random Coefficients}
    \titlerunning{Combining SG, MLMC, MLQMC for Elliptic PDEs with Random Coefficients}
    \author{Michael B.~Giles \and Frances Y.~Kuo \and Ian H.~Sloan}
    \institute{
     Michael B.~Giles (\Letter)
     \at Mathematical Institute, University of Oxford,
      Woodstock Road, Oxford, OX2 6GG, United Kingdom \\
     \email{mike.giles@maths.ox.ac.uk}
     \and
     Frances Y.~Kuo
     \at School of Mathematics and Statistics,
         University of New South Wales, Sydney NSW 2052, Australia \\
     \email{f.kuo@unsw.edu.au}
     \and
     Ian H.~Sloan
     \at School of Mathematics and Statistics,
         University of New South Wales, Sydney NSW 2052, Australia \\
     \email{i.sloan@unsw.edu.au}
    }
    \maketitle

    \index{Giles, Michael B.}
    \index{Kuo, Frances Y.}
    \index{Sloan, Ian H.}

    \abstract{Building on previous research which generalized multilevel Monte
    Carlo methods using either sparse grids or Quasi-Monte Carlo methods, this
    paper considers the combination of all these ideas applied to elliptic
    PDEs with finite-dimensional uncertainty in the coefficients. It shows the
    potential for the computational cost to achieve an $O(\varepsilon)$
    r.m.s.~accuracy to be $O(\varepsilon^{-r})$ with $r\!<\!2$, independently
    of the spatial dimension of the PDE.}

\else

    \title{Combining Sparse Grids, Multilevel MC and QMC
           for Elliptic PDEs with Random Coefficients}

   \author{Michael B.~Giles\footnote{
     Michael B.~Giles (\Letter):
     Mathematical Institute, University of Oxford,
     Woodstock Road, Oxford, OX2 6GG, United Kingdom,
     \email{mike.giles@maths.ox.ac.uk}
    }
    \and Frances Y.~Kuo\footnote{
     Frances Y.~Kuo:
     School of Mathematics and Statistics,
     University of New South Wales, Sydney NSW 2052, Australia,
     \email{f.kuo@unsw.edu.au}
    }
    \and Ian H.~Sloan\footnote{
     Ian H.~Sloan:
     School of Mathematics and Statistics,
     University of New South Wales, Sydney NSW 2052, Australia,
     \email{i.sloan@unsw.edu.au}
    }
    }

    \date{}

    \maketitle

    \begin{abstract}
    Building on previous research which generalized multilevel Monte
    Carlo methods using either sparse grids or Quasi-Monte Carlo methods, this
    paper considers the combination of all these ideas applied to elliptic
    PDEs with finite-dimensional uncertainty in the coefficients. It shows the
    potential for the computational cost to achieve an $O(\varepsilon)$
    r.m.s.~accuracy to be $O(\varepsilon^{-r})$ with $r\!<\!2$, independently
    of the spatial dimension of the PDE.
    \end{abstract}
\fi

\section{Introduction} \label{sec:intro}

There has been considerable research in recent years into the estimation
of the expected value of output functionals $P(u)$ arising from the
solution of elliptic PDEs of the form
\begin{equation} \label{eq:PDE}
 - \nabla \cdot \left(\rule{0in}{0.15in} a(\bsx,\bsy) \nabla u(\bsx,\bsy) \right) = f(\bsx),
\end{equation}
in the unit hypercube $[0,1]^d$, with homogeneous Dirichlet boundary
conditions. Here $\bsx$ represents the $d$-dimensional spatial coordinates
and the gradients are with respect to these, while $\bsy$ represents the
uncertainty. In this paper we will consider the simplest possible setting
in which we have finite $s$-dimensional uncertainty where
\[
a(\bsx,\bsy) = a_0(\bsx) + \sum_{j=1}^s y_j\, a_j(\bsx),
\]
with the $y_j$ independently and uniformly distributed on the interval
$[-\frac{1}{2}, \frac{1}{2}]$, with $0<a_{\min}\le a(\bsx,\bsy)\le
a_{\max} < \infty$ for all $\bsx$ and $\bsy$. This is the so-called
``uniform case''.

In this paper we consider several grid-based sampling methods, in all of
which the PDE \eqref{eq:PDE} is solved approximately by full or sparse
grid-based methods with respect to $\bsx$, for selected values of~$\bsy$.
We will consider both multilevel and multi-index methods
\cite{giles08,hnt15}, and compare Monte Carlo (MC) and Quasi-Monte Carlo
(QMC) methods for computing expected values with respect to $\bsy$.
We pay attention to the dependence of the computational cost on the
spatial dimension $d$, and we assume throughout this paper that the
stochastic dimension $s$ is fixed, though possibly large, and
we do not track the dependence of the cost on $s$.

As a general approach in a wide range of stochastic applications, the
multilevel Monte Carlo (MLMC) approach \cite{giles08} computes solutions
with different levels of accuracy, using the coarser solutions as a
control variate for finer solutions.  If the spatial dimension $d$ is not
too large, this can lead to an r.m.s.~accuracy of $\varepsilon$ being
achieved at a computational cost which is $O(\varepsilon^{-2})$, which is
much better than when using the standard MC method.

The earliest multilevel research on this problem was on the use of the
MLMC method for both this ``uniform case'' \cite{bsz11,hps13} and the
harder ``lognormal case'' \cite{cst13,cgst11,hps16,tsgu13}
in which $a(\bsx,\bsy)$ has a log-normal distribution with a
specified spatial covariance so that $\log a(\bsx,\bsy)$ has a
Karhunen-Lo\`eve expansion of the form $\log a(\bsx,\bsy) = \kappa_0(\bsx)
+ \sum_{j=1}^\infty y_j \sqrt{\lambda_j}\, \kappa_j(\bsx)$, where the
$y_j$ are independent with a standard normal distribution, and $\lambda_j$
and $\kappa_j(\bsx)$ are the non-decreasing eigenvalues and orthonormal
eigenfunctions of integral operator involving the covariance kernel. For
simplicity we will restrict our discussions to the uniform case in this
paper, but our results can be easily adapted for the lognormal case.

Subsequent research \cite{dkls16,gw09,hps17,ksssu15,kss15} combined the
multilevel approach with the use of QMC points, to form multilevel
Quasi-Monte Carlo (MLQMC). In the best cases, this can further reduce the
computational cost to $O(\varepsilon^{-r})$ for $r\!<\!2$.

The efficiency of both MLMC and MLQMC suffers when $d$ is large, and the
reason for this is easily understood.  Suppose the numerical
discretisation of the PDE has order of accuracy $p$, so that the error in
the output functional is $O(h^p)$, where $h$ is the grid spacing in each
coordinate direction.  To achieve an $O(\varepsilon)$ accuracy requires
$h\!=\!O(\varepsilon^{1/p})$, but if this is the grid spacing in each
direction then the total number of grid points is $O(\varepsilon^{-d/p})$.
Hence, the computational cost of performing just one calculation on the
finest level of resolution is $O(\varepsilon^{-d/p})$, and this then gives
a lower bound on the cost of the MLMC and MLQMC methods.

This curse of dimensionality is well understood, and in the case of
deterministic PDEs (i.e., without the uncertainty $\bsy$) it has been
addressed through the development of sparse grid methods \cite{bg04}. One
variant of this, the sparse combination technique, was the inspiration for
the development of the multi-index Monte Carlo (MIMC) method \cite{hnt15}.
The latter is a generalization of MLMC which in the context of
multi-dimensional PDEs uses a variety of regular grids, with differing
resolutions in each spatial direction.

In this paper we have two main contributions:
\begin{itemize}
\item we present alternative ways of combining MLMC with sparse
    grids, and discuss their relationship to the MIMC method;
\item we extend these approaches by considering the use of
    randomised QMC points, and derive the resulting computational cost
    if certain conditions are met.
\end{itemize}

The paper begins by reviewing sparse grid, MLMC/MIMC and randomised QMC
methods \cite{bg04,dks13,giles15}. Next we consider the combination of
MLMC with sparse grids, before adding randomised QMC to the combination.
In doing so, we present meta-theorems on the resulting computational cost,
based on key assumptions about the asymptotic behaviour of certain
quantities.

\section{Sparse Grid Methods} \label{sec:sg}

There are two main classes of sparse grid methods for deterministic PDEs:
sparse finite elements and the sparse combination technique \cite{bg04}.

\subsection{Sparse Finite Element Method} \label{sec:sg1}

The sparse finite element method for elliptic PDEs uses a standard
Galerkin finite element formulation but with a sparse finite element
basis.  One advantage of this approach is that most of the usual finite
element numerical analysis remains valid; the accuracy of the method can
be bounded by using bounds on the accuracy in interpolating the exact
solution using the sparse finite element basis functions.  The main
disadvantage of the approach compared to the sparse combination technique
(see the next subsection) is the difficulty of its implementation.

Following the very clear description of the method in \cite{bg99}, suppose
that we are interested in approximating the solution of an elliptic PDE in
$d$-dimensions.  For a non-negative multi-index $\GKSbell = (\ell_1,
\ell_2, \ldots, \ell_d)$, let $\calV_\GKSbell$ be the finite element space
spanned by the usual $d$-linear hat functions on a grid with spacing
$2^{-\ell_j}$ in dimension~$j$ for each $j=1,\ldots,d$. The difference
space $\calW_\GKSbell$ is defined by
\[
\calW_\GKSbell = \calV_\GKSbell \ominus \left( \bigoplus_{j=1}^d \calV_{\GKSbell-\GKSbe_j} \right)
\]
where $\GKSbe_j$ is the unit vector in direction $j$. Thus,
$\calW_\GKSbell$ has the minimal set of additional basis elements such
that
\[
\calV_\GKSbell = \calW_\GKSbell \oplus  \left( \bigoplus_{j=1}^d \calV_{\GKSbell-\GKSbe_j} \right).
\]

A sparse finite element space is then defined by $ \bigoplus_{\GKSbell \in
{\cal L}} \calW_\GKSbell, $ for some index set $\cal L$.  A simple and
near-optimal choice for a given level of accuracy is the set $\calL =
\{\GKSbell : \|\GKSbell\|_1 \!\leq\! L\}$ for some integer $L$; this is
discussed in \cite{bg99} (that paper also presents a slightly better
choice). Having defined the finite element space used for both test and
trial functions, the rest of the formulation is the standard Galerkin
finite element method. In the following, the space $\calH_1$ is the
standard Sobolev space with mixed first derivatives in $\bsx$.

\begin{theorem}[Sparse finite element method] \label{thm:sg1}
For fixed $\bsy$, if the PDE \eqref{eq:PDE} is solved using the sparse
finite element method with the index set specified by $\|\GKSbell\|_1
\!\leq\! L$, then the computational cost is $O(L^{d-1}\, 2^{L})$.
Moreover, the $\calH_1$ solution accuracy is $O(2^{-L})$ if the solution
$u$ has sufficient mixed regularity, and the accuracy of simple output
functionals $P$ (such as smoothly weighted averages of the solution) is
$O(2^{-2L})$. Hence, the cost to achieve a functional accuracy of
$\varepsilon$ is $O(\varepsilon^{-1/2} |\log \varepsilon|^{d-1})$.
\end{theorem}

\begin{proof}
The cost and $\calH_1$ solution accuracy are proved in \cite{bg99,gsz92}.
The super-convergence for output functionals is an immediate consequence
of adjoint-based error analysis \cite{gs02}.
\end{proof}

\subsection{Sparse Combination Method} \label{sec:sg2}

The sparse combination method combines the results of separate
calculations on simple tensor product grids with different resolutions in
each coordinate direction \cite{gsz92}. For a given output functional $P$
and multi-index $\GKSbell = (\ell_1,\ldots,\ell_d)$, let $P_\GKSbell$
denote the approximate output functional obtained on a grid with spacing
$2^{-\ell_j}$ in direction $j$ for each $j=1,\ldots,d$. For convenience,
we define $P_\GKSbell:=0$ if any of the indices in $\GKSbell$ is negative.

The backward difference in the $j^{\rm th}$ dimension is defined as $
\Delta_j P_\GKSbell := P_\GKSbell - P_{\GKSbell - \GKSbe_j}, $ and we
define the $d$-dimensional mixed first difference as
\[
\bsDelta P_\GKSbell := \left( \prod_{j=1}^d \Delta_j \right) P_\GKSbell.
\]
For an arbitrary multi-index $\GKSbell'$, it can be shown that
\begin{equation} \label{eq:summation}
P_{\GKSbell'} = \sum_{0 \;\leq\; \GKSbell \;\leq\; \GKSbell'} \bsDelta P_\GKSbell,
\end{equation}
where the multi-index inequality $\GKSbell \leq \GKSbell'$ is applied
element-wise (i.e.~$\ell_j\!\leq\! \ell'_j, \forall j$). Taking the limit
as $\GKSbell'\rightarrow \boldsymbol{\infty}$ (i.e.~$\ell'_j \rightarrow
\infty, \forall j$) gives
\begin{equation} \label{eq:full_combination}
  P = \sum_{\GKSbell \;\geq\; 0} \bsDelta P_\GKSbell.
\end{equation}
The sparse combination method truncates the summation to a finite index
set, with a simple and near-optimal choice again being $\|\GKSbell\|_1
\!\leq\! \ell$. This gives the approximation
\begin{equation}
 P_{\ell} := \sum_{\|\GKSbell\|_1 \leq \ell} \bsDelta P_\GKSbell,
\label{eq:sparse_combination}
\end{equation}
where we are slightly abusing notation by distinguishing between the
original $P_\GKSbell$ with a multi-index subscript (in bold type with a
tilde underneath), and the new $P_{\ell}$ on the left-hand side of this
equation with a scalar subscript (which is not in bold).

If we now define
\begin{equation}
S_\ell := \sum_{ \|\GKSbell\|_1 = \ell} P_\GKSbell
\label{eq:sparse_sum}
\end{equation}
and the backward difference $\Delta S_\ell := S_\ell - S_{\ell-1}$, then
it can be shown \cite{reisinger13} that
\[
P_\ell = \Delta^{d-1} S_\ell = \sum_{k=0}^{d-1} (-1)^k \binom{d\!-\!1}{k}\ S_{\ell-k}.
\]
Hence, the computation of
$P_\ell$ requires $O(\ell^{d-1})$ separate computations, each on a grid
with $O(2^\ell)$ grid points. This leads to the following theorem.

\begin{theorem}[Sparse combination method] \label{thm:sg2}
For fixed $\bsy$, if the PDE \eqref{eq:PDE} is solved using the sparse
combination method with the index set specified by $\|\GKSbell\|_1
\!\leq\! L$, then the computational cost is $O(L^{d-1}\, 2^{L})$.
Moreover, if the underlying PDE approximation has second order accuracy
and the solution $u$ has sufficient mixed regularity, then $\bsDelta
P_\GKSbell$ has magnitude $O(2^{-2\|\GKSbell\|_1})$ so the error in $P_L$
is $O(L^{d-1} \, 2^{-2L})$. Hence, the cost to achieve a functional
accuracy of $\varepsilon$ is $O(\varepsilon^{-1/2} |\log
\varepsilon|^{3\,(d-1)/2})$.
\end{theorem}

\begin{proof}
For the results on the cost and accuracy see \cite{reisinger13}. The cost
result is an immediate consequence.
\end{proof}

\section{MLMC and MIMC}

\subsection{MLMC}

\label{sec:MLMC}

The multilevel Monte Carlo (MLMC) idea is very simple. As explained in a
recent review article \cite{giles15}, given a sequence $P_\ell,\, \ell
\!=\! 0, 1, \ldots$ of approximations of an output functional $P$, with
increasing accuracy and cost as $\ell$ increases, and defining
$P_{-1}:=0$, we have the simple identity
\[
\bbE[P] = \sum_{\ell=0}^\infty \bbE[ \Delta P_\ell], ~~~
\Delta P_\ell := P_\ell - P_{\ell-1}.
\]
The summation can be truncated to
\begin{equation}
\bbE[P]\ \approx\ \bbE[P_L]\ =\ \sum_{\ell=0}^L \bbE[\Delta P_\ell],
\label{eq:telescoping}
\end{equation}
with $L$ chosen to be sufficiently large to ensure that the weak error
$\bbE[P\!-\!P_L]$ is acceptably small. Each of the expectations on the
r.h.s.~of \eqref{eq:telescoping} can be estimated independently using
$N_\ell$ independent samples so that the MLMC estimator is
\begin{equation} \label{eq:MLMC}
Y = \sum_{\ell=0}^L Y_\ell, ~~~~
Y_\ell = \frac{1}{N_\ell} \sum_{i=1}^{N_\ell} \Delta P_\ell^{(i)}.
\end{equation}
The computational savings comes from the fact that on the finer levels
$\Delta P_\ell$ is smaller and has a smaller variance, and therefore fewer
samples $N_\ell$ are required to accurately estimate its expected value.

The optimal value for $N_\ell$ on level $\ell = 0, 1, \ldots, L$ can be
estimated by approximately minimising the cost for a given overall
variance.  This results in the following theorem which is a slight
generalization of the original in \cite{giles08}.

\begin{theorem}[MLMC]
\label{thm:MLMC} Let $P$ denote an output functional, and let $P_\ell$
denote the corresponding level $\ell$ numerical approximation. Suppose
there exist independent estimators $Y_\ell$ of $\bbE[\Delta P_\ell]$ based
on $N_\ell$ Monte Carlo samples and positive constants $\alpha, \beta,
\gamma, c_1, c_2, c_3$, with
$\alpha\!\geq\!\frac{1}{2}\min(\beta,\gamma)$, such that
\begin{itemize}
\item[i)] ~ $\displaystyle \left|\rule{0in}{0.13in} \bbE[P_\ell \!-\!
    P] \right| \longrightarrow 0 ~~ \mbox{as} ~~ \ell \longrightarrow
    \infty $, \vspace{0.05in}
\item[ii)] ~ $\displaystyle \left|\bbE[\Delta P_\ell] \right|\ \leq\
    c_1\, 2^{-\alpha\, \ell} $ \vspace{0.05in},
\item[iii)] ~ $\displaystyle \bbE[Y_\ell] = \bbE[\Delta P_\ell] $,
    \vspace{0.05in}
\item[iv)] ~ $\displaystyle \bbV[Y_\ell]\ \leq\ c_2\,
    N_\ell^{-1}\,2^{-\beta\, \ell} $, \vspace{0.05in}
\item[v)] ~ $\displaystyle {\rm cost}(Y_\ell)\ \leq\ c_3\,
    N_\ell\,2^{\gamma\, \ell} $.
\end{itemize}
Then there exists a positive constant $c_4$ such that for any $\varepsilon
\!<\! e^{-1}$ there are values $L$ and $N_\ell$ for which the MLMC
estimator
\eqref{eq:MLMC} achieves the 
mean-square-error bound $
\bbE [ (Y - \bbE[P])^2] < \varepsilon^2$ with the computational cost bound
\[
{\rm cost}(Y)
\leq \left\{\begin{array}{ll}
c_4\, \varepsilon^{-2}              ,    & ~~ \beta>\gamma, \\[0.1in]
c_4\, \varepsilon^{-2} |\log \varepsilon|^2,    & ~~ \beta=\gamma, \\[0.1in]
c_4\, \varepsilon^{-2-(\gamma\!-\!\beta)/\alpha}, & ~~ \beta<\gamma.
\end{array}\right.
\]
\end{theorem}

The proof of this theorem uses a constrained optimisation approach to
optimise the number of samples $N_\ell$ on each level.  This treats the
$N_\ell$ as real variables, and then the optimal value is rounded up to
the nearest integer.  This rounding up improves the variance slightly, so
that we still achieve our target mean-square-error accuracy, but it also
increases the cost by at most one sample per level.  This additional cost
is dominated by the cost of one sample on the finest level, which is
$O(\varepsilon^{-\gamma/\alpha})$ since the weak convergence condition
requires that the finest level satisfies $2^{-\alpha L} = O(\varepsilon)$.
The condition in the theorem that $\alpha\geq
\frac{1}{2}\min(\beta,\gamma)$ ensures that this additional cost is
negligible compared to the main cost.

When applied to our model elliptic PDE, if one uses a tensor product grid
with spacing $2^{-\ell}$ in each direction, then if the numerical
discretisation has second order accuracy it gives $\alpha\!=\!2$ and
$\beta\!=\!4$, while with an ideal multigrid solver the cost is at best
proportional to the number of grid points which is $2^{d\ell}$ so
$\gamma\!=\!d$. Hence, the cost is $O(\varepsilon^{-r})$ where
$r=\max(2,d/2)$, except for $d\!=\!4$ for which $\beta\!=\!\gamma$ and
hence there is an additional $|\log \varepsilon|^2$ factor.  It is the
dependence on $d$ which will be addressed by incorporating sparse grid
methods.

\subsection{MIMC}

The multi-index Monte Carlo (MIMC) method  \cite{hnt15} is inspired by the
sparse combination technique. Starting from \eqref{eq:full_combination},
if each of the $\Delta P_\GKSbell$ is now a random variable due to the
random coefficients in the PDE, we can take expectations of each side and
truncate the sum to give
\begin{equation}
\bbE[P] \,\approx\,
\bbE[P_{L}] \,=\, \sum_{\|\GKSbell\|_1 \leq L} \bbE[\bsDelta P_\GKSbell].
\label{eq:sparse_combination_2}
\end{equation}
This is now very similar to the telescoping sum \eqref{eq:telescoping} in
MLMC, with the difference that the levels are now labelled by
multi-indices, so allowing different discretizations in different
directions.  We can independently estimate each of the expectations on the
r.h.s.~of \eqref{eq:sparse_combination_2} using a number of independent
samples $N_\GKSbell$ so that the MIMC estimator is
\begin{equation} \label{eq:MIMC}
Y = \sum_{\|\GKSbell\|_1 \leq L} Y_\GKSbell,~~~~
Y_\GKSbell = \frac{1}{N_\GKSbell} \sum_{i=1}^{N_\GKSbell} \Delta P_\GKSbell^{(i)}.
\end{equation}
The numbers $N_\GKSbell$ are optimised to minimise the cost of achieving a
certain desired variance or mean-square-error.

The original paper \cite{hnt15} considers much more general circumstances:
the different indices in $\GKSbell$ are not limited to the spatial
discretizations in $\bsx$ but can also involve quantities such as the
number of particles in a system, or the number of terms in a
Karhunen--Lo\`eve expansion (arising from dimension truncation in the
stochastic variables $\bsy$). Here in the isotropic PDE case, in which the
behaviour in each space dimension is similar, this leads to the following
theorem.

\begin{theorem}[MIMC]
\label{thm:MIMC} Let $P$ denote an output functional, and for each
multi-index $\GKSbell$ let $P_\GKSbell$ denote the approximate output
functional indexed by $\GKSbell$. Suppose for each multi-index $\GKSbell$
there exist independent estimators $Y_\GKSbell$ of $\bbE[\bsDelta
P_\GKSbell]$ based on $N_\GKSbell$ Monte Carlo samples and positive
constants $\alpha, \beta, \gamma, c_1, c_2, c_3$, with
$\alpha\!\geq\!\frac{1}{2} \beta$, such that
\begin{itemize}
\item[i)] ~ $\displaystyle \left|\rule{0in}{0.13in} \bbE[P_\GKSbell
    \!-\! P] \right| \longrightarrow 0 ~~ \mbox{as} ~~ \GKSbell
    \longrightarrow \boldsymbol{\infty} $ ($\ell_j\to\infty\;, \forall
    j$), \vspace{0.05in}
\item[ii)] ~ $\displaystyle \left|\rule{0in}{0.13in} \bbE[\bsDelta
    P_\GKSbell] \right| \ \leq\ c_1\, 2^{-\alpha \|\GKSbell\|_1} $,
    \vspace{0.05in}
\item[iii)] ~ $\displaystyle \bbE[Y_\GKSbell]\ = \bbE[\bsDelta
    P_\GKSbell] $, \vspace{0.05in}
\item[iv)] ~ $\displaystyle\bbV[Y_\GKSbell]\ \leq\ c_2\,N_\ell^{-1}\,
    2^{-\beta\|\GKSbell\|_1} $, \vspace{0.05in}
\item[v)] ~ $\displaystyle {\rm cost}(Y_\GKSbell) \leq\ c_3\,N_\ell\,
    2^{\gamma\|\GKSbell\|_1} $.
\end{itemize}
Then there exists a positive constant $c_4$ such that for any $\varepsilon
\!<\! e^{-1}$ there are values $L$ and $N_\GKSbell$ for which the MIMC
estimator
\eqref{eq:MIMC} achieves the 
mean-square-error bound $
\bbE [ (Y - \bbE[P])^2] < \varepsilon^2$ with the computational cost bound
\[
{\rm cost}(Y) \leq\ \left\{\begin{array}{ll}
c_4\, \varepsilon^{-2}\,                ,    & ~~ \beta>\gamma, \\[0.1in]
c_4\, \varepsilon^{-2}\, |\log \varepsilon|^{e_1},    & ~~ \beta=\gamma, \\[0.1in]
c_4\, \varepsilon^{-2-(\gamma-\beta)/\alpha}\, |\log \varepsilon|^{e_2}, & ~~ \beta<\gamma,
\end{array}\right.
\]
where
\[
\begin{array}{lll}
e_1 = 2d, &
e_2 = (d\!-\!1)\,(2\!+\!(\gamma\!-\!\beta)/\alpha), &
\mbox{if} ~~ \alpha \!>\! \textstyle\frac{1}{2}\beta, \\[0.05in]
e_1 = \max(2d, 3(d\!-\!1)), &
e_2 = (d\!-\!1)\,(1\!+\!\gamma/\alpha), &
\mbox{if} ~~ \alpha \!=\! \textstyle\frac{1}{2}\beta.
\end{array}
\]
\end{theorem}
\begin{proof}
This is a particular case of the more general analysis in
\cite[Theorem 2.2]{hnt15}.
\end{proof}

In the case of MIMC, there are $O(L^{d-1})$ multi-indices on the finest
level on which $\|\GKSbell\|_1 = L$.  Hence the finest level is determined
by the constraint $L^{d-1} 2^{-\alpha L} = O(\varepsilon)$, and the
associated cost is $O(\varepsilon^{-\gamma/\alpha} |\log
\varepsilon|^{(d-1)(1+\gamma/\alpha)})$.  Given the assumption that
$\alpha \geq \frac{1}{2}\beta$, this is not asymptotically bigger than the
main cost except when $\alpha\!=\!\frac{1}{2}\beta$, in which case it is
responsible for the $e_2$ and the $3(d\!-\!1)$ component in the maximum in
$e_1$.

When applied to our model elliptic PDE, if one uses a tensor product grid
with spacing $2^{-\ell_j}$ in the $j^{\rm th}$ direction, and a numerical
discretisation with second order accuracy, then we are likely to get
$\alpha\!=\!2$ and $\beta\!=\!4$ if the solution has sufficient mixed
regularity \cite{reisinger13}. (Note that this is a much stronger
statement than the $\alpha\!=\!2$, $\beta\!=\!4$ in the previous section;
taking the case with $d\!=\!3$ as an example, with grid spacing $h_1, h_2,
h_3$ in the three dimensions, Section \ref{sec:MLMC} requires only that
$\Delta P_\ell = O(h^2)$ when all three spacings are equal to $h$, whereas
in this section we require the product form $\bsDelta P_\GKSbell =
O(h_1^2\,h_2^2\,h_3^2)$ which is much smaller when $h_1,h_2,h_3 \ll 1$.)
With an ideal multigrid solver, the cost is proportional to
$2^{\|\GKSbell\|_1}$, so $\gamma\!=\!1$.  Since $\beta\!>\!\gamma$, the
cost would then be $O(\varepsilon^{-2})$, regardless of the value of $d$.

\section{Randomised QMC and MLQMC}

\subsection{Randomised QMC Sampling}

A randomized QMC method with $N$ deterministic points and $R$
randomization steps approximates an $s$-dimensional integral over the unit
cube $[-\frac{1}{2},\frac{1}{2}]^s$ as follows
\[
  I := \int_{[-\frac{1}{2},\frac{1}{2}]^s} g(\bsy)\,\rd\bsy
  \quad\approx\quad \overline{Q} := \frac{1}{R} \sum_{k=1}^R Q_k,
  \qquad
  Q_k = \frac{1}{N} \sum_{i=1}^N g(\bsy^{(i,k)}).
\]

For the purpose of this paper it suffices that we introduce briefly just a
simple family of randomized QMC methods -- randomly shifted lattice rules.
We have
\[
  \bsy^{(i,k)} = \left\{ \frac{i\bsz}{N} + \bsDelta^{(k)}\right\}
  - \tfrac{\boldsymbol{1}}{\boldsymbol{2}},
\]
where $\bsz\in \bbN^s$ is known as the generating vector; $\bsDelta^{(1)},
\ldots, \bsDelta^{(R)}\in (0,1)^s$ are $R$ independent random shifts; the
braces indicate that we take the fractional part of each component in the
vector; and finally we subtract $\frac{1}{2}$ from each component of the
vector to bring it into $[-\frac{1}{2},\frac{1}{2}]^s$.

Randomly shifted lattice rules provide unbiased estimators of the
integral. Indeed, it is easy to verify that $\bbE_\bsDelta[\overline{Q}] =
\bbE_\bsDelta[Q_k] = I$, where we introduced the subscript $\bsDelta$ to
indicate that the expectation is taken with respect to the random shifts.
In some appropriate function space setting for the integrand function $g$,
it is known (see e.g., \cite{dks13}) that good generating vectors $\bsz$
can be constructed so that the variance or mean-square-error satisfies
$\bbV_\bsDelta[\overline{Q}] = \bbE_\bsDelta[(\overline{Q}-I)^2] \le
C_\delta\, R^{-1}\,N^{-2(1-\delta)}$, for some $\delta\in (0,1/2]$ with
$C_\delta$ independent of the dimension $s$. In practical computations, we
can estimate the variance by $\bbV_\bsDelta[\overline{Q}] \approx
\sum_{k=1}^R (Q_k - \overline{Q})^2/[R(R-1)]$. Typically we take a large
value of $N$ to benefit from the higher QMC convergence rate and use only
a relatively small $R$ (e.g., $20$--$50$) for the purpose of estimating
the variance.

There are other randomization strategies for QMC methods. For example,
we can combine any digital net such as Sobol$'$ sequences or interlaced
polynomial lattice rules with digital shift or Owen scrambling, to get an
unbiased estimator with variance close to $O(N^{-2})$ or $O(N^{-3})$. We
can also apply randomization to a higher order digital net to achieve
$O(N^{-p})$ for $p\!>\!2$ in an appropriate function space setting for smooth
integrands. For detailed reviews of these results see see e.g.,
\cite{dks13}.

\subsection{MLQMC} \label{sec:MLQMC}

As a generalization of \eqref{eq:MLMC}, the multilevel Quasi-Monte Carlo
(MLQMC) estimator is
\begin{equation} \label{eq:MLQMC}
Y = \sum_{\ell=0}^L Y_\ell, ~~~~
Y_\ell = \frac{1}{R_\ell} \sum_{k=1}^{R_\ell}
\left(\frac{1}{N_\ell} \sum_{i=1}^{N_\ell} \Delta P_\ell^{(i,k)} \right).
\end{equation}
Later in Theorem~\ref{thm:MLQMC} we will state the corresponding
generalization of Theorem~\ref{thm:MLMC}.

The use of QMC instead of MC in a multilevel method was first considered
in \cite{gw09} where numerical experiments were carried out for a number
of option pricing problems and showed convincingly that MLQMC improves
upon MLMC.  A meta-theorem similar to the MLMC theorem was proved in
\cite{gn13}.  A slightly sharper version of the theorem, eliminating some
$\log(\varepsilon)$ factors, will be stated and proved later in
\S\ref{sec:MIQMC}.

MLQMC methods have been combined with finite element discretizations for
the PDE problems in \cite{dkls16,ksssu15,kss15}. The paper \cite{kss15}
studied the uniform case for the same elliptic PDE of this paper with
randomly shifted lattice rules (which yield up to order $2$ convergence in
the variance); the paper \cite{dkls16} studied the uniform case for
general operator equations with deterministic higher order digital nets;
the paper \cite{ksssu15} studied the lognormal case with randomly shifted
lattice rules. A key analysis which is common among these papers is the
required mixed regularity estimate of the solution involving both $\bsx$
and $\bsy$, see \cite{kn16} for a survey of the required analysis in a
unified framework.

\section{Combining Sparse Grids and MLMC}

After this survey of the three component technologies, sparse grid
methods, MLMC and MIMC, and randomised QMC samples, the first novel
observation in this paper is very simple: MIMC is not the only way in
which MLMC can be combined with sparse grid methods.

An alternative is to use the standard MLMC approach, but with samples
which are computed using sparse grid methods. The advantage of this is
that it can be used with either sparse finite elements or the sparse
combination technique.

\subsection{MLMC with Sparse Finite Element Samples}

In Theorem~\ref{thm:MLMC}, if $P_\ell$ is computed using sparse finite
elements as described in Section~\ref{sec:sg1} based on grids with index
set $\|\GKSbell\|_1\leq \ell$, and if the accuracy and cost are as given
in Theorem~\ref{thm:sg1}, then we obtain $\alpha\!=\!2\!-\!\delta$,
$\beta\!=\!4\!-\!\delta$, and $\gamma\!=\!1\!+\!\delta$ for any
$0\!<\!\delta \!\ll\! 1$. Here $\delta$ arises due to the effect of some
additional powers of $\ell$. So $\beta \!>\!\gamma$ and therefore the
computational cost is $O(\varepsilon^{-2})$.

Recall that with the full tensor product grid we had $\alpha\!=\!2$,
$\beta\!=\!4$, and $\gamma = d$. Hence the improvement here is in the
removal of the dependence of the cost parameter $\gamma$ on~$d$.

\subsection{MLMC with Sparse Combination Samples}

The aim in this section is to show that the MIMC algorithm is very
similar to MLMC using sparse combination  samples.

Suppose we have an MIMC application which satisfies the conditions of
Theorem~\ref{thm:MIMC}.  For the MLMC version, we use
\eqref{eq:sparse_combination} to define the $P_\ell$ in
Theorem~\ref{thm:MLMC}. Since
\begin{equation}
\bbE[ \Delta P_\ell ] = \sum_{\|\GKSbell\|_1 = \ell} \bbE[\bsDelta P_\GKSbell],
\label{eq:MIMC_equiv}
\end{equation}
the two algorithms have exactly the same expected value if the finest
level for each is given by $\|\GKSbell\|_1 \!=\! L$ for the same value of
$L$.  The difference between the two algorithms is that MIMC independently
estimates each of the expectations on the r.h.s.~of \eqref{eq:MIMC_equiv},
using a separate estimator $Y_\GKSbell$ for each $\bbE[\bsDelta
P_\GKSbell]$ with independent samples of $\bsy$, whereas MLMC with sparse
combination samples estimates the expectation on the l.h.s., using the
combination
\[
Y_\ell = \sum_{\|\GKSbell\|_1 = \ell} Y_\GKSbell,
\]
with the $Y_\GKSbell$ all based on the same set of $N_\ell$ random samples
$\bsy$.

There are no more than $(\ell\!+\!1)^{d-1}$ terms in the summation in
(\ref{eq:MIMC_equiv}), so if the cost of $Y_\GKSbell$ for MIMC is
$O(N_\GKSbell 2^{\gamma \ell})$ when $\|\GKSbell\|_1 \!=\! \ell$, then the
cost of the sparse combination estimator $Y_\ell$ for MLMC is $O(N_\ell
\ell^{d-1}2^{\gamma \ell}) = o(N_\ell 2^{(\gamma+\delta) \ell})$, for any
$0\!<\!\delta \!\ll\! 1$.

Likewise,
\[
\left|\,\bbE[Y_\ell]\, \right| \leq \sum_{\|\GKSbell\|_1 = \ell} \left|\,\bbE[Y_\GKSbell]\, \right|,
\]
so if $|\,\bbE[Y_\GKSbell]\,| = O(2^{-\alpha\ell})$ when $\|\GKSbell\|_1
\!=\! \ell$, then $|\,\bbE[Y_\ell]\,| = o(2^{-(\alpha-\delta)\ell})$ for
any $0\!<\!\delta \!\ll\! 1$.

Furthermore, Jensen's inequality gives
\begin{eqnarray*}
\bbV\left[ Y_\ell \right]
 \ =\ \bbE\left[ (Y_\ell - \bbE[Y_\ell])^2 \right]
   &=& \bbE\bigg[ \bigg(\sum_{\|\GKSbell\|_1 = \ell} (Y_\GKSbell - \bbE[Y_\GKSbell]) \bigg)^2 \bigg]
\\ &\leq&
(\ell\!+\!1)^{d-1}\!\! \sum_{\|\GKSbell\|_1 = \ell} \bbE \left[(Y_\GKSbell - \bbE[Y_\GKSbell])^2\right]
\\ &=&
(\ell\!+\!1)^{d-1}\!\! \sum_{\|\GKSbell\|_1 = \ell} \bbV [Y_\GKSbell],
\end{eqnarray*}
so if $\bbV[Y_\GKSbell] \!=\! O(N^{-1}_\GKSbell 2^{- \beta \ell})$, then
$\bbV[Y_\ell] \!=\! o(N^{-1}_\ell\, 2^{- (\beta-\delta) \ell})$, for any
$0\!<\!\delta \!\ll\! 1$.

This shows that the $\alpha$, $\beta$, $\gamma$ values for the
MLMC algorithm using the sparse combination samples are almost
equal to the $\alpha$, $\beta$, $\gamma$ for the MIMC method,
which leads to the following lemma.

\begin{lemma}
If a numerical method satisfies the conditions for the MIMC
Theorem~\ref{thm:MIMC}, then the corresponding MLMC estimator with sparse
combination samples will have a cost which is $O(\varepsilon^{-2})$, if
$\beta \!>\! \gamma$, and $o(\varepsilon^{-2-(\gamma-\beta)/\alpha) -
\delta})$, $\forall\, 0\!<\!\delta\!\ll\! 1$, if $\beta \!\leq\! \gamma$.
\end{lemma}

As with MLMC with sparse finite element samples, the key thing here is
that the level $\ell$ MLMC samples use a set of grids in which the number
of grid points is $O(2^{\|\GKSbell\|_1}) \!=\! O(2^\ell)$. That is why the
$\gamma$ values for MIMC and MLMC are virtually identical.

If there is substantial cancellation in the summation, it is possible that
$\bbV[Y_\ell]$ could be very much smaller than the $\bbV[Y_\GKSbell]$ for
each of the $\GKSbell$ for which $\|\GKSbell\|_1 \!=\! \ell$. However, we
conjecture that this is very unlikely, and therefore we are not suggesting
that the MLMC with sparse combination samples is likely to be better than
MIMC. The point of this section is to show that it cannot be significantly
worse. In addition, this idea of combining MLMC with sparse grid samples
works for sparse finite elements for which there seems to be no natural
MIMC extension.

\subsection{Nested MLMC}

Another alternative to MIMC is nested MLMC.  To illustrate this in 2D,
suppose we start by using a single level index $\ell_1$ to construct a
standard MLMC decomposition
\[
\bbE[P]\ \approx\ \bbE[P_{L_1}]\ =\ \sum_{\ell_1=0}^{L_1} \bbE[\Delta P_{\ell_1}].
\]
Now, for each particular index $\ell_1$ we can take $\bbE[\Delta
P_{\ell_1}]$ and perform a secondary MLMC expansion with respect to a
second index $\ell_2$ to give
\[
\bbE[\Delta P_{\ell_1}]\ \approx\ \sum_{\ell_2=0}^{L_2}
\bbE[Q_{\ell_1,\ell_2} - Q_{\ell_1,\ell_2-1}],
\]
with $Q_{\ell_1,-1}\!:=\!0$. If we allow $L_2$ to possibly depend on the
value of $\ell_1$, this results in an approximation which is very similar
to the MIMC method,
\[
\bbE[P]\ \approx\ \sum_{\GKSbell \in \calL}
\bbE\left[ Q_{\ell_1,\ell_2} - Q_{\ell_1,\ell_2-1} \right],
\]
with the summation over some finite set of indices $\calL$. In contrast to
the MIMC method, here $Q_{\ell_1,\ell_2} - Q_{\ell_1,\ell_2-1}$ is not
necessarily expressible in the cross-difference form $\bsDelta P_\GKSbell$
used in MIMC.  Thus, this method is a generalization of MIMC.

This approach is currently being used in two new research projects. In one
project, the second expansion is with respect to the precision of floating
point computations; i.e.~half, single or double precision. This follows
ideas presented in section 10.2 of \cite{giles15} and also in
\cite{bswohrkk14}. In the other project \cite{hg17}, the second expansion
uses Rhee \& Glynn's randomised multilevel Monte Carlo method \cite{rg15}
to provide an unbiased inner estimate in a financial nested expectation
application.

\section{MLQMC and MIQMC} \label{sec:MIQMC}

The next natural step is to replace the Monte Carlo sampling with
randomised QMC sampling to estimate $\bbE[\Delta P_\ell]$ or
$\bbE[\bsDelta P_\GKSbell]$.

\subsection{MLQMC (continued from \S\ref{sec:MLQMC})}

In the best circumstances, using $N_\ell$ QMC deterministic points with
$R_\ell = R$ randomisation steps to estimate $\bbE[\Delta P_\ell]$ gives a
variance (with respect to the randomisation in the QMC points) which is
$O(R^{-1} N_\ell^{-p} 2^{-\beta \ell})$, with $p\!>\!1$. This leads to the
following theorem which generalizes Theorem~\ref{thm:MLMC}.


\begin{theorem}[MLQMC]
\label{thm:MLQMC} Let $P$ denote an output functional, and let $P_\ell$
denote the corresponding level $\ell$ numerical approximation. Suppose
there exist independent estimators $Y_\ell$ of $\bbE[\Delta P_\ell]$ based
on $N_\ell$ deterministic QMC points and $R_\ell = R$ randomization steps,
and positive constants $\alpha, \beta, \gamma, c_1, c_2, c_3$, $p$, with
$p>1$ and $\alpha\!\geq\!\frac{1}{2}\beta$, such that
\begin{itemize}
\item[i)] ~ $\displaystyle \left|\rule{0in}{0.13in} \bbE[P_\ell \!-\!
    P] \right| \longrightarrow 0 ~~ \mbox{as} ~~ \ell \longrightarrow
    \infty $, \vspace{0.05in}
\item[ii)] ~ $\displaystyle \left|\bbE[\Delta P_\ell] \right|\ \leq\
    c_1\, 2^{-\alpha\, \ell} $ \vspace{0.05in},
\item[iii)] ~ $\displaystyle \bbE_\bsDelta[Y_\ell] = \bbE[\Delta
    P_\ell] $, \vspace{0.05in}
\item[iv)] ~ $\displaystyle \bbV_\bsDelta[Y_\ell]\ \leq\ c_2\,
    R^{-1}\,N_\ell^{-p}\,2^{-\beta\, \ell} $, \vspace{0.05in}
\item[v)] ~ $\displaystyle {\rm cost}(Y_\ell)\ \leq\ c_3\, R\,N_\ell\,
    2^{\gamma\, \ell} $.
\end{itemize}
Then there exists a positive constant $c_4$ such that for any $\varepsilon
\!<\! e^{-1}$ there are values $L$ and $N_\ell$ for which the MLQMC
estimator
\eqref{eq:MLQMC} achieves the 
mean-square-error bound $
\bbE_\bsDelta [ (Y - \bbE[P])^2] < \varepsilon^2$ with the computational
cost bound
\[ {\rm cost}(Y) \leq \left\{\begin{array}{ll}
c_4\, \varepsilon^{-2/p}              ,    & ~~ \beta>p \gamma, \\[0.1in]
c_4\, \varepsilon^{-2/p} |\log \varepsilon|^{(p+1)/p},    & ~~ \beta=p\gamma, \\[0.1in]
c_4\, \varepsilon^{-2/p-(p\gamma\!-\!\beta)/(p\alpha)}, & ~~ \beta<p\gamma.
\end{array}\right.
\]
\end{theorem}

\begin{proof}
We omit the proof here because the theorem can be interpreted as a special
case of Theorem~\ref{thm:MIQMC} below for which we will provide an outline
of the proof.
\end{proof}

\subsection{MIQMC}

As a generalization of \eqref{eq:MIMC}, the MIQMC estimator is
\begin{equation} \label{eq:MIQMC}
Y = \sum_{\|\GKSbell\|_1 \leq L} Y_\GKSbell,~~~~
Y_\GKSbell = \frac{1}{R_\GKSbell} \sum_{k=1}^{R_\GKSbell}
\left( \frac{1}{N_\GKSbell} \sum_{i=1}^{N_\GKSbell}  \Delta P_\GKSbell^{(i,k)} \right),
\end{equation}
where $Y_\GKSbell$ is an estimator for $\bbE[\bsDelta P_\GKSbell]$ based
on $N_\GKSbell$ deterministic QMC points and $R_\GKSbell$ randomization
steps.

Suppose that $Y_\GKSbell$ has variance and cost given by
$\bbV_\bsDelta[Y_\GKSbell] = N_\GKSbell^{-p}v_\GKSbell$ and ${\rm
cost}(Y_\GKSbell)=N_\GKSbell\, c_\GKSbell$. The variance and total cost of
the combined estimator $Y$ are
\[
  \bbV_\bsDelta[Y] \!=\!\! \sum_{\| \GKSbell \|_1 \leq L} N_\GKSbell^{-p} v_\GKSbell,~~~~
  {\rm cost}(Y) \!=\! \sum_{\| \GKSbell \|_1 \leq L} N_\GKSbell\, c_\GKSbell.
\]
Treating the $N_\GKSbell$ as real numbers, the cost can be minimised for a
given total variance by introducing a Lagrange multiplier and minimising
${\rm cost}(Y) \!+\! \lambda \bbV_\bsDelta[Y]$, which gives
\[
N_\GKSbell = \left( \frac{\lambda\, p\, v_\GKSbell}{c_\GKSbell} \right)^{1/(p+1)}.
\]
Requiring $\bbV_\bsDelta[Y] \!=\! \frac{1}{2}\varepsilon^2$ to achieve a
target accuracy determines the value of $\lambda$ and then the total cost
is
\[
{\rm cost}(Y) =
 (2\, \varepsilon^{-2})^{1/p} \left(
\sum_{\| \GKSbell \|_1 \leq L} \left( c_\GKSbell^p v_\GKSbell \right)^{1/(p+1)}
\right)^{(p+1)/p}.
\]

This outline analysis shows that the behaviour of the product
$c_\GKSbell^p v_\GKSbell$ as $\GKSbell \rightarrow \infty$ is critical. If
$c_\GKSbell = O(2^{\gamma \ell})$ and $v_\GKSbell = O(2^{-\beta\ell})$
where $\ell = \|\GKSbell \|_1$, then $c_\GKSbell^p v_\GKSbell =
O(2^{(p\gamma - \beta)\ell})$.

If $\beta \!>\! p \gamma$, then the total cost is dominated by the
contributions from the coarsest levels, and we get a total cost which is
$O(\varepsilon^{-2/p})$.

If $\beta \!=\! p \gamma$, then all levels contribute to the total cost,
and it is $O(L^{d(p+1)/p} \varepsilon^{-2/p})$.

If $\beta \!<\! p \gamma$, then the total cost is dominated by the
contributions from the finest levels, and we get a total cost which is
$O(L^{(d-1)(p+1)/p}\, \varepsilon^{-2/p}\, 2^{(p\gamma-\beta)L/p})$.

To complete this analysis, we need to know the value of $L$ which is
determined by the requirement that the square of the bias is no more than
$\frac{1}{2}\varepsilon^2$.  This can be satisfied by ensuring that
\[
{\rm bias}(Y) := \sum_{\| \GKSbell \|_1 > L} \left| \bbE[\bsDelta P_\GKSbell] \right|
\ \leq\ \varepsilon / \sqrt{2}.
\]
If $|\bbE[\bsDelta P_\GKSbell]| = O(2^{-\alpha \|\GKSbell\|_1})$, then the
contributions to ${\rm bias}(Y)$ come predominantly from the coarsest of
the levels in the summation (i.e.~$\|\GKSbell\|_1 = L+1$), and hence ${\rm
bias}(Y) \!=\! O(L^{d-1}2^{-\alpha L})$. The bias constraint then gives $
L^{d-1} 2^{-\alpha L} \!=\! O(\varepsilon) $ and hence $L \!=\! O(|\log
\varepsilon|)$.

As discussed after the MLMC and MIMC theorems, the values for $N_\GKSbell$
need to be rounded up to the nearest integers, incurring an additional
cost which is $O(\varepsilon^{-\gamma/\alpha} |\log
\varepsilon|^{(d-1)(1+\gamma/\alpha)})$. If $\alpha \!>\! \frac{1}{2}
\beta$ it is always negligible compared to the main cost, but it can
become the dominant cost when $\alpha \!=\! \frac{1}{2} \beta$ and
$\beta\!\leq\!p\gamma$. This corresponds to the generalization of Cases C
and D in Theorem 2.2 in the MIMC analysis in \cite{hnt15}.

This outline analysis leads to the following theorem in which we make
various assumptions and then draw conclusions about the resulting cost.

\begin{theorem}[MIQMC] \label{thm:MIQMC}
Let $P$ denote an output functional, and for each multi-index $\GKSbell$
let $P_\GKSbell$ denote the approximate output functional indexed by
$\GKSbell$. Suppose for each multi-index $\GKSbell$ there exist
independent estimators $Y_\GKSbell$ of $\bbE[\bsDelta P_\GKSbell]$ based
on $N_\GKSbell$ deterministic QMC samples and $R_\ell = R$ randomization
steps, and positive constants $\alpha, \beta, \gamma, c_1, c_2, c_3$, $p$,
with $p\!>\!1$ and $\alpha\!\geq\!\frac{1}{2} \beta$, such that
\begin{itemize}
\item[i)] ~ $\displaystyle \left|\rule{0in}{0.13in} \bbE[P_\GKSbell
    \!-\! P] \right| \longrightarrow 0 ~~ \mbox{as} ~~ \GKSbell
    \longrightarrow \boldsymbol{\infty} $  ($\ell_j\to\infty\;,
    \forall j$), \vspace{0.05in}
\item[ii)] ~ $\displaystyle \left|\rule{0in}{0.13in} \bbE[\bsDelta
    P_\GKSbell] \right| \ \leq\ c_1\, 2^{-\alpha \|\GKSbell\|_1} $
    \vspace{0.05in}
\item[iii)] ~ $\displaystyle \bbE_\Delta[Y_\GKSbell]\ = \bbE[\bsDelta
    P_\ell] $ \vspace{0.05in}
\item[iv)] ~ $\displaystyle \bbV_\Delta[Y_\GKSbell]\ \leq\ c_2\,
    R^{-1} N_\GKSbell^{-p}\, 2^{-\beta\|\GKSbell\|_1} $
    \vspace{0.05in}
\item[v)] ~ $\displaystyle {\rm cost}(Y_\GKSbell)\ \leq\ c_3\, R\,
    N_\GKSbell\ 2^{\gamma\|\GKSbell\|_1} $.
\end{itemize}
Then there exists a positive constant $c_4$ such that for any $\varepsilon
\!<\! e^{-1}$ there are values $L$ and $N_\GKSbell$ for which the MIQMC
estimator
\eqref{eq:MIQMC} achieves the 
mean-square-error bound $
\bbE_\bsDelta [ (Y - \bbE[P])^2] < \varepsilon^2$ with the computational
cost bound
\[
{\rm cost}(Y) \ \leq\ \left\{\begin{array}{ll}
c_4\, \varepsilon^{-2/p}\,                , & ~~ \beta > p\gamma, \\[0.1in]
c_4\, \varepsilon^{-2/p}\, |\log \varepsilon|^{e_1}, & ~~ \beta = p\gamma, \\[0.1in]
c_4\, \varepsilon^{-2/p\, -\, (p\gamma-\beta)/p\alpha}
 \, |\log \varepsilon|^{e_2}, & ~~ \beta < p\gamma,
\end{array}\right.
\]
where
\[
\begin{array}{ll}
e_1 = d(p\!+\!1)/p, \quad
e_2 = (d\!-\!1)((p\!+\!1)/p +(p\gamma\!-\!\beta)/p\alpha), & \quad
\mbox{if} ~~ \alpha \!>\! \textstyle\frac{1}{2}\beta, \\[0.05in]
e_1 = \max(d(p\!+\!1)/p,\, (d\!-\!1)(1 \!+\!\gamma/\alpha)), \quad
e_2 = (d\!-\!1)(1 \!+\! \gamma/\alpha), &      \quad
\mbox{if} ~~ \alpha \!=\! \textstyle\frac{1}{2}\beta.
\end{array}
\]
\end{theorem}

\begin{proof}
Based on the outline proof which indicates how to specify the near optimal
number of QMC points for each level, the detailed proof follows the same
lines as the main MIMC Theorem 2.2 in \cite{hnt15}.
\end{proof}

The key observation here is that the dimension $d$ does not appear in the
exponent for $\varepsilon$ in the cost bounds, so it is a significant
improvement over the MLQMC result in which the cost is of the form
$\varepsilon^{-r}$ with $r=\max(2/p, d/2)$, which limits the multilevel
benefits even for $d\!=\!3$ if $p\!>\!4/3$.

It is interesting to compare the cost given by this theorem with that
given by the MIMC Theorem~\ref{thm:MIMC}. If $\beta\!>\!p\gamma$, then the
use of QMC improves the cost from $O(\varepsilon^{-2})$ to
$O(\varepsilon^{-2/p})$. This is because the dominant costs in this case
are on the coarsest levels where many points have to be sampled, and
therefore QMC will provide substantial benefits. On the other hand, if
$\beta \!<\! \gamma$ then both approaches give a cost of approximately
$O(\varepsilon^{-\gamma/\alpha})$ because in this case the dominant costs
are on the finest levels, and on the finest levels the optimal number of
QMC points is $O(1)$, which is why the additional cost of rounding up to
the nearest integer often dominates the main cost. Hence the use of QMC
points is almost irrelevant in this case. Fortunately, we expect that the
favourable case $\beta\!>\!p\gamma$ is likely to be the more common one.
It is clearly the case in our very simple elliptic model with
$\beta\!=\!4$ and $\gamma\!=\!1$.

\section{Concluding Remarks}

In this paper we began by summarizing the meta-theorems for MLMC and
MIMC in a common framework for elliptic PDEs with random coefficients,
where we applied full or sparse grid methods with respect to the
spatial variables $\bsx$ and used MC sampling for computing expected
values with respect to the stochastic variables $\bsy$.

Following this, our novel contributions were
\begin{itemize}
\setlength{\itemsep}{0.5em}
\item
showing that, in this context, MIMC is almost equivalent to
the use of MLMC with sparse combination samples;
\item introducing the idea of a) MLMC with sparse finite element or
    sparse combination samples, and b) nested MLMC, as other
    alternatives to MIMC;
\item deriving the corresponding meta-theorems for MLQMC and MIQMC in
    this context, concluding that the computational cost to achieve
    $O(\varepsilon)$ r.m.s.~accuracy can be reduced to
    $O(\varepsilon^{-r})$ with $r<2$ independent of the spatial
    dimension $d$.
\end{itemize}

Natural extensions of the results in this paper include allowing the
different indices in $\GKSbell$ to cover also different levels of
dimension truncation in the stochastic variables $\bsy$, as well as
providing verifications of the precise parameters $\alpha$, $\beta$,
$\gamma$ and $p$ for specific PDE applications.

\begin{acknowledgement}
The authors acknowledge the support of the Australian Research Council
under the projects FT130100655 and DP150101770.
\end{acknowledgement}

%
\bibliographystyle{spmpsci}

%

\end{document}

%% file: macros.tex
\usepackage{makeidx}         
\usepackage{graphicx}        
\usepackage[bottom]{footmisc}

\usepackage{latexsym}
\usepackage{amsmath}
\usepackage{amsfonts}
\usepackage{amssymb}
\usepackage{bm}

\usepackage{url}
\usepackage{algorithm}
\usepackage{algorithmic}
\usepackage[misc,geometry]{ifsym}

\newcommand{\bsx}{{\boldsymbol{x}}}
\newcommand{\bsy}{{\boldsymbol{y}}}
\newcommand{\bsz}{{\boldsymbol{z}}}

\newcommand{\bsDelta}{{\boldsymbol{\Delta}}}

\newcommand{\rd}{{\mathrm{d}}}

\newcommand{\bbE}{{\mathbb{E}}}

\newcommand{\bbN}{{\mathbb{N}}}

\newcommand{\bbV}{{\mathbb{V}}}

\DeclareSymbolFont{bbold}{U}{bbold}{m}{n}
\DeclareSymbolFontAlphabet{\mathbbold}{bbold}

\newcommand{\calH}{{\mathcal{H}}}

\newcommand{\calL}{{\mathcal{L}}}

\newcommand{\calV}{{\mathcal{V}}}
\newcommand{\calW}{{\mathcal{W}}}





%% file: mcqmc-GKS-20171019.bbl
\begin{thebibliography}{99.}%

\bibitem{bsz11} 
 Barth, A., Schwab, Ch., Zollinger, N.:
\newblock Multi-level Monte Carlo finite element method for elliptic PDEs with
          stochastic coefficients.
\newblock {\em Numerische Mathematik}, 119(1):123-161 (2011)

\bibitem{bswohrkk14}
 Brugger, C., de Schryver, C., Wehn, N., Omland, S., Hefter, M., Ritter, K., Kostiuk, A., Korn, R.:
\newblock Mixed precision multilevel Monte Carlo on hybrid computing systems.
\newblock Proceedings of the Conference on Computational
          Intelligence for Financial Engineering and Economics, IEEE
          (2014)

\bibitem{bg99} 
 Bungartz, H.-J., Griebel, M.:
\newblock A note on the complexity of solving {P}oisson's equation for spaces
  of bounded mixed derivatives.
\newblock {\em Journal of Complexity}, 15(2):167--199 (1999)

\bibitem{bg04} 
 Bungartz, H.-J., Griebel, M.:
\newblock Sparse grids.
\newblock {\em Acta Numerica}, 13:1--123 (2004)

\bibitem{cst13} 
 Charrier, J., Scheichl, R., Teckentrup, A.:
\newblock Finite element error analysis of elliptic {PDE}s with random
  coefficients and its application to multilevel {M}onte {C}arlo methods.
\newblock {\em SIAM Journal on Numerical Analysis}, 51(1):322--352 (2013)

\bibitem{cgst11} 
 Cliffe, K.A., Giles, M.B., Scheichl, R., Teckentrup, A.:
\newblock Multilevel {M}onte {C}arlo methods and applications to elliptic
  {PDE}s with random coefficients.
\newblock {\em Computing and Visualization in Science}, 14(1):3--15 (2011)

\bibitem{dkls16} 
 Dick, J., Kuo, F.Y., Le~Gia, Q.T., Schwab, Ch.:
\newblock Multi-level higher order {QMC} {P}etrov--{G}alerkin discretization for affine
  parametric operator equations.
 \newblock {\em SIAM Journal on Numerical Analysis}, 54(4):2541--2568 (2016)

\bibitem{dks13} 
 Dick, J., Kuo, F.Y., Sloan, I.H.:
\newblock High-dimensional integration: the quasi-{M}onte {C}arlo way.
\newblock {\em Acta Numerica}, 22:133--288 (2013)

\bibitem{gn13} 
 Gerstner, T., Noll, M.:
\newblock Randomized multilevel quasi-{M}onte {C}arlo path simulation.
\newblock In {\em Recent Developments in Computational Finance}, pages
          349--372. World Scientific (2013)

\bibitem{giles08} 
 Giles, M.B.:
\newblock Multilevel {M}onte {C}arlo path simulation.
\newblock {\em Operations Research}, 56(3):607--617 (2008)

\bibitem{giles15} 
 Giles, M.B.:
\newblock Multilevel {M}onte {C}arlo methods.
\newblock {\em Acta Numerica}, 24:259--328 (2015)

\bibitem{gs02} 
 Giles, M.B., S{\"u}li, E.:
\newblock Adjoint methods for {PDE}s: {\it a posteriori} error analysis and
  postprocessing by duality.
\newblock {\em Acta Numerica}, 11:145--236 (2002)

\bibitem{gw09} 
 Giles, M.B., Waterhouse, B.J.:
\newblock Multilevel quasi-{M}onte {C}arlo path simulation.
\newblock In {\em Advanced Financial Modelling}, Radon Series on Computational
  and Applied Mathematics, pages 165--181. De Gruyter (2009)

\bibitem{gsz92} 
 Griebel, M., Schneider, M., Zenger, C.:
\newblock A combination technique for the solution of sparse grid problems.
\newblock In P.~de~Groen and P.~R.~Beauwens, editors, {\em Iterative Methods in
  Linear Algebra}. IMACS (1992)

\bibitem{hnt15} 
 Haji-Ali, A.-L., Nobile, F., Tempone, R.:
\newblock Multi {I}ndex {M}onte {C}arlo: when sparsity meets sampling.
\newblock {\em Numerische Mathematik}, 132:767--806 (2016)

\bibitem{hg17} 
 Haji-Ali, A.-L., Giles, M.B.:
\newblock MLMC for Value-at-Risk.
\newblock Presentation at {\em International Conference on
          Monte Carlo Methods and Applications}, 2017.\hfill\\
\newblock (URL: http://people.maths.ox.ac.uk/hajiali/assets/files/hajiali-mcm2017-var.pdf)

\bibitem{hps13} 
 Harbrecht, H., Peters, M., Siebenmorgen, M.:
\newblock On multilevel quadrature for elliptic stochastic partial differential
  equations.
\newblock In {\em Sparse Grids and Applications}, volume~88 of {\em Lecture
  Notes in Computational Science and Engineering}, pages 161--179. Springer
  (2013)

\bibitem{hps16} 
 Harbrecht, H., Peters, M., Siebenmorgen, M.:
\newblock Multilevel accelerated quadrature for PDEs with log-Normally
          distributed diffusion coefficient.
\newblock {\em SIAM/ASA Journal on Uncertainty Quantification}, 4(1):520--551
(2016)

\bibitem{hps17} 
 Harbrecht, H., Peters, M., Siebenmorgen, M.:
\newblock On the quasi-Monte Carlo method with Halton points for elliptic PDEs with log-normal
diffusion.
\newblock {\em Mathematics of Computation}, 86:771--797 (2017)

\bibitem{kn16} 
 Kuo, F.Y., Nuyens, D.:
\newblock Application of quasi-Monte Carlo methods to elliptic PDEs with
random diffusion coefficients -- a survey of analysis and implementation.
\newblock {\em Foundations of Computational Mathematics}, 16(6):1631--1696
(2016)

\bibitem{ksssu15} 
 Kuo, F.Y., Scheichl, R., Schwab, Ch., Sloan, I.H., Ullmann, E.:
\newblock Multilevel quasi-{M}onte {C}arlo methods for lognormal diffusion
  problems.
\newblock {\em Mathematics of Computation}, 86:2827--2860 (2017)

\bibitem{kss15} 
 Kuo, F.Y., Schwab, Ch., Sloan, I.H.:
\newblock Multi-level quasi-{M}onte {C}arlo finite element methods for a class
  of elliptic partial differential equations with random coefficients.
\newblock {\em Foundations of Computational Mathematics}, 15(2):411--449
(2015)

\bibitem{rg15}
 Rhee, C.-H., Glynn, P.W.:
\newblock Unbiased estimation with square root convergence for {SDE} models.
\newblock {\em Operations Research}, 63(5):1026--1043 (2015)

\bibitem{reisinger13} 
 Reisinger, C.:
\newblock Analysis of linear difference schemes in the sparse grid combination
  technique.
\newblock {\em IMA Journal of Numerical Analysis}, 33(2):544--581 (2013)

\bibitem{tsgu13} 
 Teckentrup, A., Scheichl, R., Giles, M.B., Ullmann, E.:
\newblock Further analysis of multilevel {M}onte {C}arlo methods for elliptic
  {PDE}s with random coefficients.
\newblock {\em Numerische Mathematik}, 125(3):569--600 (2013)
\end{thebibliography}
